\documentclass[a4paper,twoside]{article}
\usepackage{amsthm,amscd}
\usepackage{amsmath,amssymb,amsfonts}
\usepackage[T1]{fontenc}
\usepackage{tikz}
\usetikzlibrary{trees}
\usepackage{pdfpages}

\usepackage[utf8]{inputenc}
\usepackage[english]{babel}
\usepackage{graphicx}               
\usepackage{color}                  
\RequirePackage[colorlinks=true,citecolor=blue,linkcolor=blue]{hyperref}
\usepackage{hyperref}
\usepackage[absolute]{textpos} 
\usepackage{multicol}
\usepackage{array}

\usepackage{bbm}
\newcommand{\ind}{\mathbbm{1}}

\newtheorem{theorem}{Theorem}[section]

\theoremstyle{definition}
\newtheorem{definition}[theorem]{Definition}

\newtheorem{rmk}[theorem]{Remark}

\def \leq{\leqslant}
\def \geq{\geqslant}

\newcommand{\eqdef}{\mathrel{=}:}

\usepackage{authblk}
 
\newcommand{\beq} {\begin{eqnarray*}}
\newcommand{\eeq} {\end{eqnarray*}}
\newcommand{\trm} {\textrm}

\def \R{\mathbb{R}}

\def \ZZ{\mathbf{Z}}

\def \E{\mathbb{E}}
\def \P{\mathbb{P}}
\def \Var{\hbox{{\rm Var}}}

\def \Cov{\hbox{{\rm Cov}}}

\newcommand{\z}{\mathbf{z}}

\newcommand{\cvloi}{\overset{\mathcal{L}}{\underset{N\to\infty}{\longrightarrow}}}

\usepackage[utf8]{inputenc}
\usepackage[]{geometry}
\usepackage[mathcal]{euscript}
\usepackage{bm,url}                     
\usepackage{amsfonts}
\usepackage{amssymb}
\usepackage{amsmath}
\usepackage{amsthm}
\usepackage{graphicx}               
\usepackage{epstopdf}
\usepackage[english]{babel}
\usepackage{amsfonts}
\usepackage{color}
\usepackage{array}
\usepackage{mathrsfs}
\usepackage{hyperref}

\newcommand{\M}{M}

\newcommand{\tr}{\mbox{tr}}

\title{Sensitivity analysis in general metric spaces} 

\author[1]{Fabrice Gamboa}
\author[2]{Thierry Klein} 
\author[3]{Agn\`es Lagnoux}
\author[4]{Leonardo Moreno}
\affil[1]{Institut de Math\'ematiques de Toulouse and ANITI; UMR5219. Université de Toulouse; CNRS. UT3, F-31062 Toulouse, France.}
\affil[2]{Institut de Math\'ematiques de Toulouse; UMR5219. Université de Toulouse; ENAC - Ecole Nationale de l'Aviation Civile, Universit\'e de Toulouse, France}
\affil[3]{Institut de Math\'ematiques de Toulouse; UMR5219. Université de Toulouse; CNRS. UT2J, F-31058 Toulouse, France.}
\affil[4]{Departamento de Métodos Cuantitativos, FCEA, Universidad de la República, Uruguay}

\begin{document}

\maketitle
\begin{abstract}
Sensitivity indices are commonly used to quantity the relative influence of any specific group of input variables on the output of a computer code. In this paper, we introduce new sensitivity indices adapted to outputs valued in general metric spaces. This new class of indices encompasses the classical ones; in particular, the so-called Sobol indices and the Cramér-von-Mises indices. Furthermore, we provide asymptotically Gaussian estimators of these indices based on U-statistics. Surprisingly, we prove the asymptotic normality straightforwardly. 
Finally, we illustrate this new procedure on a toy model and on two real-data examples.
\end{abstract}


\section{Introduction}

In the last decades, the use of computer code experiments to model physical phenomena has become a recurrent task for many applied researchers and engineers. In such simulations, it is crucial to understand the global influence
of  one or several input variables on  the  output of the system. When considering these inputs as random elements, this   problem  is  generally  called  (global)
sensitivity analysis. We refer, for example to \cite{rocquigny2008uncertainty} or \cite{saltelli-sensitivity}
 for an overview on practical aspects of sensitivity analysis.

One of the most popular indicator to quantify the influence of some inputs is the so-called Sobol index. This index was first introduced in \cite{pearson1915partial} and then considered by \cite{sobol2001global}. It is well tailored when the output space is $\R$. It compares using the so-called Hoeffding decomposition (see \cite{Hoeffding48}) the conditional variance of the output (knowing some of the input variables) with the total variance of the output. Many different estimation procedures of the Sobol indices have been proposed and studied in the literature. Some are based on Monte-Carlo or quasi Monte-Carlo design of experiments (see \cite{Kucherenko2017different,owen2} and references therein for more details). More recently a method based on nested Monte-Carlo \cite{GODA201763} has been developed. In particular, an efficient estimation of the Sobol indices can be performed through the so-called Pick-Freeze method. For the description of this method and its theoretical study (consistency, central limit theorem, concentration inequalities and Berry- Esseen bounds), we refer to \cite{pickfreeze,janon2012asymptotic} and references therein. Some other estimation procedures are based on different designs of experiments using for example polynomial chaos expansions (see \cite{Sudret2008global} and the reference therein for more details).

The case of vectorial outputs was first studied in \cite{lamboni2011multivariate} and tackled using principal component analysis. In \cite{GJKL14}, the authors recover the indices proposed in \cite{lamboni2011multivariate} and showed that in some sense they are the only reasonable generalization of the classical Sobol indices in dimension greater than 2. Moreover, they provide the theoretical study of the Pick-Freeze estimators and extend their definition to the case of outputs valued in a separable Hilbert space.

Since Sobol indices are based on the variance through the Hoeffding decomposition, they only quantify the input influence on the mean value of the computer code. Many authors proposed other ways to compare the conditional distribution of the output knowing some of the inputs to the distribution of the output. In \cite{Owen12,ODC13}, the authors considered higher moments to define new indices, whereas in \cite{borgonovo2007,borgonovo2011moment, borgonovo2016common,DaVeiga13}, divergences or distances between measures are used. In \cite{FKR13,kala2019quantile}, the authors used contrast functions to build goal-oriented indices. Although these works defined nice theoretical indices, the existence of an efficient statistical estimation procedure is still in most cases an open question. The case of vectorial-valued computer codes is considered in \cite{GKL18} where a sensitivity index based on the whole distribution of the output utilizing the Cram\'er-von-Mises distance is defined. The authors showed that the Pick-Freeze estimation procedure can be used providing an asymptotically Gaussian estimator of the index. This scheme requires $3N$ evaluations of the output code and leads to a convergence rate of order $\sqrt N$. This approach has been generalized in \cite{FGM2017}, where the authors considered computer codes valued in a compact Riemannian manifold. Once again, they used the Pick-Freeze scheme to provide a consistent estimator of their index, requiring $4N$ evaluations of the output. Unfortunately, no central limit theorem was proved.
 
In this work, we build general indices for a code valued in a metric space and we provide an asymptotically Gaussian estimator based on U-statistics requiring only  $2N$ evaluations of the output code while keeping a convergence rate of  $\sqrt N$. In addition, we explain that all the indices studied in \cite{FGM2017,GJKL14,pickfreeze,GKL18,janon2012asymptotic} can be seen as particular cases of our framework. Hence, we improve the estimation scheme of \cite{GKL18} and \cite{FGM2017} by reducing to $2N$ the number of evaluations of the code. Last but not least, using the results of Hoeffding \cite{Hoeffding48} on U-statistics, the asymptotic normality is proved straightforwardly.

The paper is organized as follows. Section \ref{sec:setting} is dedicated to the definition of the new indices and the presentation of their estimation via U-statistics. In Section \ref{sec:appli}, we recover the classical indices classically used in global sensitivity analysis. Furthermore, we extend the work of \cite{FGM2017} and establish the central limit theorem that was not yet proved. We illustrate the procedure in Section \ref{sec:num} on a toy example and on two real-data models. The first application is about the Gaussian plume model and consists in quantifying the sensitivity of the contaminant concentration with respect
to some input parameters. Second, an elliptical  differential partial equation of diffusive transport type is considered. In this setting, we proceed to the singular value decomposition of the solution and we perform a sensitivity analysis of the orthogonal matrix produced by the decomposition with respect to the equation parameters. Finally, some conclusions are given in Section \ref{sec:concl}.

\section{General setting}\label{sec:setting}

\subsection{The Cramér-von-Mises indices}

The main idea of Cramér-von-Mises indices is to compare the conditional cumulative distribution function (c.d.f.) to the unconditional one via some distance. The supremum norm is used in \cite{borgonovo289probabilistic} while in \cite{GKL18} the $L^2$-norm is chosen. The previous approaches are global. A local approach is considred in \cite{luyi2012moment}.

Here, we consider a measurable function $f$ (black-box code) defined on $E= E_1\times E_2\times \dots \times E_p$ and valued in a separable metric space $(\mathcal X,d)$. Here, $(E_1,\mathcal A_1)$, $\cdots$, $(E_p,\mathcal A_p)$ are measurable spaces. The output denoted by $Z$ is  then given by
\begin{equation}\label{def:model}
Z=f(X_1,\dots, X_p),
\end{equation}
where $X_i$ is a random element of $E_i$ and $X_1,\dots, X_p$ are assumed to be mutually independent. Naturally, we assume that all the random variables are defined on the same probability space $(\Omega, \mathcal A,\P)$ and $\omega \mapsto (X_1(\omega),  \ldots , X_p(\omega))$ is a measurable application from $\Omega$ to $E$. 

In \cite{GKL18}, the authors studied, for $\mathcal X=\R^k$, global sensitivity indices of $Z$ with respect to the inputs $X_1$,$\ldots$, $X_p$ based on its whole distribution (instead of considering only its second moment as done usually via the so-called Sobol indices). Those indices are based on the Cramér-von-Mises distance. To do so, they introduced a family of test functions parameterized by a single index $t=(t_1,\ldots,t_k)\in \R^k$ and defined by
\begin{align*}
T_t(Z)=\ind_{\{Z\leqslant t\}}=\ind_{\{Z_1\leqslant t_1, \ldots, Z_k\leqslant t_k\}}.
\end{align*}

More precisely, let $\bf u$ be a subset of $I_p=\{1,\ldots,p\}$ and let $\sim \textbf{u}$ be its complementary in $I_p$ ($\sim \textbf{u}=I_p\setminus \textbf{u}$). We define $X_{\bf u}=(X_i)_{i\in \textbf{u}}$. For $t=(t_1,\ldots, t_k)\in\R^k$, let also $F$ be the distribution function of $Z$:
\[
F(t)=\P\left(Z\leqslant t\right)=\E\left[\ind_{\{Z\leqslant t\}}\right],
\]
and $F^{\textbf{u}}$ be the conditional distribution function of $Z$ conditionally on $X_{\textbf{u}}$:
\[
F^{\textbf{u}}(t)=\P\left(Z\leqslant t|X_{\textbf{u}}\right)=\E\left[\ind_{\{Z\leqslant t\}}|X_{\textbf{u}}\right].
\]
Obviously, for any $t\in \R^k$, $F^{\textbf{u}}(t)$ is a random variable depending only on $t$ and $X^{\textbf u}$, the expectation of which is $\E\left[F^{\textbf{u}}(t)\right]=F(t)$. Since for any fixed $t\in \R^{k}$, $T_t(Z)$ is a real-valued random variable, we can perform its Hoeffding decomposition with respect to $\textbf u$ and $\sim \textbf u$:
\begin{align*}
T_t(Z)=F(t)+((F^u(t)-F(t))+(F^{\sim u}(t)-F(t)) + R_t(X^u,X^{\sim u}),
\end{align*}
where 
\[
R_t(X^u,X^{\sim u})=T_t(Z)-\E[Y(t)T_t(Z)]-\left(\E[T_t(Z)|X_{\textbf{u}}]-\E[T_t(Z)]\right)-\left(\E[T_t(Z)|X_{\sim \textbf{u}}]-\E[T_t(Z)]\right)
\]
leading to 
\begin{align}
\Var(T_t(Z))&=F(t)(1-F(t))\nonumber\\
&= \E\left[\left(F^{\textbf{u}}(t)-F(t)\right)^{2}\right]+\E\left[\left(F^{\sim \textbf{u}}(t)-F(t)\right)^{2}\right]+\Var(R_t(X^u,X^{\sim u})). \label{eq:decomp_var}
\end{align}

Then, the Cram\'er-von-Mises index is obtained by integrating over $t$  with respect to the distribution of the output code $Z$ and normalizing by the integrated total variance :
\begin{align}\label{CVM}
S_{2,CVM}^{\textbf{u}}=\frac{\int_{\R^{k}}\E\left[\left(F(t)-F^{\textbf{u}}(t)\right)^2\right]dF(t)}{\int_{\R^k} F(t)(1-F(t))dF(t)}.
\end{align}

In this example, the collection of the test functions $T_t(Z)=\ind_{\{Z\leqslant t\}}$ ($t\in \R^k$) is parameterized by a single vectorial  parameter $t$. Since the knowledge of the c.d.f.\ of $Z$: $F(t)=\E[\ind_{\{Z\leqslant t\}}]=\P(Z\leqslant t)$ characterizes its distribution, the index $S_{2,CVM}^{\textbf{u}}$ depends
as expected on the whole distribution of the output computer code. 
Using the Pick-Freeze methodology, the authors of \cite{GKL18} proposed an estimator which requires $3N$ evaluations of the output code leading to a convergence rate of $\sqrt N$.

This approach has been generalized in \cite{FGM2017} to compact Riemannian manifolds replacing the indicator function of half-spaces $\ind_{\{Z\leqslant t\}}$ parameterized by $t$
by the indicator function of balls $\ind_{\{Z\in \widetilde B(a_1,a_2)\}}$ indexed by two parameters $a_1$ and $a_2$. In this last work, $\widetilde B(a_1,a_2)$ stands for the ball whose center is the middle point between $a_1$ and $a_2$ with radius $\overline{a_1a_2}/2$. Therein a consistent estimation scheme based on $4N$ evaluations of the function is proposed. Nevertheless, the convergence rate of the estimator is not studied.

Now we aim at generalizing this methodology to any separable metric spaces and to any  class of test functions parameterized by a fixed number of elements of the metric space. 

\subsection{The general metric space sensitivity indices}\label{ssec:index}

Recall that $Z$ lives in the space $\mathcal X$. Generalizing the previous approach, we consider a family of test functions parameterized by  $m\geqslant 1$ elements of $\mathcal X$. For any $a=(a_i)_{i=1,\dots,m}\in \mathcal X^m$, we consider the test functions
 \[
\begin{matrix}
& \mathcal X^{m}\times \mathcal X & \to & \R\\
& (a,x) & \mapsto & T_a(x).\\
\end{matrix}
\]
We assume that $T_a(\cdot{})\in L^2(\P^{\otimes m}\otimes \P)$ where $\P$ denotes the distribution of $Z$.
Performing the Hoeffding decomposition on each test function $T_a(\cdot{})$ and then integrating with respect to $a$ using $\P^{\otimes m}$ leads to the definition of our new index. 

\begin{definition}
The \emph{general metric space sensitivity index}  with respect to $\textbf{u}$ is defined by
 \begin{align}\label{eq:GIM}
S_{2,GMS}^{\textbf{u}}=\frac{\int_{\mathcal X^m}\E\left[\left(\E[T_a(Z)]-\E[T_a(Z)|X_{\textbf{u}}]\right)^{2}\right]d\P^{\otimes m}(a)}{\int_{\mathcal X^m} \Var(T_a(Z))d\P^{\otimes m}(a)},
\end{align}
with $X_{\bf u}=(X_i)_{i\in \textbf{u}}$.
\end{definition}
Observe that the different contributions $S_{2,GMS}^{\textbf{u}}$, $S_{2,GMS}^{\sim \textbf{u}}$ and the integrated remaining term (see \eqref{eq:decomp_var}) sum to 1.

\paragraph{Particular examples}

By convention, when the test functions $T_a$ do not depend on $a$, we set $m=0$. 

\begin{enumerate}
\item For $\mathcal X=\R$, $m=0$, and $T_a$ given by $T_a(x)=x$,  one recovers the classical Sobol indices (see \cite{sobol1993,sobol2001global}). In this case, it appears that the parameterized test functions do not depend on the parameter $a$.  For $\mathcal X=\R^k$ and $m=0$, one can recover the index defined for vectorial outputs in \cite{GJKL14,lamboni2011multivariate} by extending \eqref{eq:GIM}. 
\item For $\mathcal X=\R^k$, $m=1$, and $T_a$ given by $T_a(x)=\ind_{\{x\leqslant a\}}$,  one recovers the index based on the Cram\'er-von-Mises distance defined in \cite{GKL18} and recalled in \eqref{CVM}. 
\item Consider that $\mathcal X=\mathcal M$ is a manifold, $m=2$ and $T_a$ is given by $T_a(x)=\ind_{\{x\in \widetilde B(a_1,a_2)\}}$, where $\widetilde B(a_1,a_2)$ stands for the ball whose center is the middle point between $a_1$ and $a_2$ with radius $\overline{a_1a_2}/2$. Here, one recovers the index  defined in \cite{FGM2017}. 
\end{enumerate}

\begin{rmk} 
The previous two first examples can be seen as particular cases of what is called Common Rationale in \cite{borgonovo2016common}. More precisely, the first-order Sobol index with respect to $X_i$ corresponds to the index $\eta_i$ in \cite[Equation (4)]{borgonovo2016common} while the Cramér-von-Mises index with respect to $X_i$ is based on the distance between the c.d.f.\ $F$ of $Z$ and its conditional version $F^i$ with respect to $X_i$. Actually, in our construction, as soon as the class of test functions $T_a$ characterizes the distribution, the index becomes a particular case of the Common Rationale.

Analogously, the authors of \cite{borgonovo2016common} also consider as particular cases the expectation of the $L^1$-distance between the p.d.f.\ of $Z$ and its conditional version with respect to $X_i$ (index $\delta_i$ in \cite[Equation (12)]{borgonovo2016common} and the expectation of the $L^\infty$-distance between
$F$ of $Z$ and its conditional version $F^i$ (index $\beta_i$ in \cite[Equation (13)]{borgonovo2016common}). Notice that  the integration in $\delta_i$ is done with respect to the Lebesgue measure whereas the integration in our general metric space sensitivity index $S_{2,GMS}^{\textbf{u}}$ in \eqref{eq:GIM} is done with respect to the distribution of the output $Z$. The benefit is twofold. First, the integral always exists. Second, such an integration weights the support of the output
distribution.
\end{rmk}

\subsection{Estimation procedure via U-statistics}\label{ssec:est}

Following the so-called Pick-Freeze scheme, let $X^{\textbf{u}}$ be the random  vector such
that $X^{\textbf{u}}_i=X_i$ if $i\in \textbf{u}$ and $X^{\textbf{u}}_i=X'_i$ if $i\notin \textbf{u}$ where $X'_i$ is an independent copy of $X_i$. Then, setting 
\begin{align}\label{def:Yv}
Z^{\textbf{u}}= f(X^{\textbf{u}}),
\end{align}
 a direct computation leads to the following relationship  (see, e.g., \cite{janon2012asymptotic}):
\begin{align*}
\Var (\mathbb{E}[T_a(Z)|X_{\textbf{u}}]) =\Cov\left(T_a(Z),T_a(Z^{\textbf{u}})\right).
\end{align*}
Now let us define $\ZZ=(Z,Z^{\textbf{u}})^{\top}$ and  consider $(m+2)$ i.i.d.\ copies of $\ZZ$ denoted by $(\ZZ_i,i=1,\dots,m+2)$. In the sequel, $\P^{\textbf{u}}_2$ stands for the law of $\ZZ=(Z,Z^{\textbf{u}})^{\top}$. Setting $A=(Z_1,\dots,Z_m)$. Then the integrand in the numerator of \eqref{eq:GIM} rewrites as 
\begin{align*}
\E\left[\left(\E[T_A(Z)]-\E[T_A(Z)|X_{\textbf{u}}]\right)^{2}\right]&=\E_A\left[\Var_{X_{\textbf{u}}} (\mathbb{E}_{Z_{m+1}}[T_A(Z_{m+1})|X_{\textbf{u}}]) \right]\nonumber\\
&=\E_A\left[\Cov_{\ZZ_{m+1}}(T_A(Z_{m+1}),T_A(Z_{m+1}^{\textbf{u}}))\right].
\end{align*}
Here the notation $\E_{Z}$ (resp. $\Var_Z$ and $\Cov_Z$) stands for the expectation (resp. the variance and the covariance) with respect to the law of the random variable $Z$.

Now, for any $1\leqslant i\leqslant m+2$, we let $\z_i=(z_i,z_i^{\textbf{u}})^{\top}$ and we define
\begin{align*}
&\Phi_1(\z_1,\dots,\z_{m+1})= T_{z_1,\dots,z_m}(z_{m+1})T_{z_1,\dots,z_m}(z_{m+1}^{\textbf{u}})\\
&\Phi_2(\z_1,\dots,\z_{m+2})= T_{z_1,\dots,z_m}(z_{m+1})T_{z_1,\dots,z_m}(z_{m+2}^{\textbf{u}})\\
&\Phi_3(\z_1,\dots,\z_{m+1})= T_{z_1,\dots,z_m}(z_{m+1})^2\\
&\Phi_4(\z_1,\dots,\z_{m+2})= T_{z_1,\dots,z_m}(z_{m+1})T_{z_1,\dots,z_m}(z_{m+2}).
\end{align*}
Further, we set 
\begin{align}\label{def:m}
M(1)=M(3)=m+1 \quad \textrm{and} \quad  M(2)=M(4)=m+2
\end{align} 
and we define, for $j=1,\dots,4$,
\begin{align}\label{def:I}
&I(\Phi_j)=\int_{\mathcal X^{M(j)}}\Phi_j(\z_1,\dots,\z_{M(j)})d\P_2^{u,\otimes M(j)}(\z_1\dots,\z_{M(j)}).
\end{align}
Finally, we introduce  the application $\Psi$ from $\R^4$ to $\R$ defined by
\begin{equation}\label{def:psi}
\begin{matrix}
\Psi:& \mathcal \R^4 & \to & \R\\
& (x,y,z,t) & \mapsto & \frac{x-y}{z-t}.
\end{matrix}
\end{equation}
Then, $S_{2,GMS}^{\textbf{u}}$ can be rewritten as 
\begin{align}
S_{2,GMS}^{\textbf{u}}&=
\Psi\left(I(\Phi_1),I(\Phi_2),I(\Phi_3),I(\Phi_4) \right). \label{eq:USTAT2}
\end{align}
The previous expression of $S_{2,GMS}^{\textbf{u}}$ will allow to perform easly its estimation. Following Hoeffding \cite{Hoeffding48},  we replace  the functions $\Phi_1,\Phi_2$, $\Phi_3$ and $\Phi_4$ by their symmetrized version $\Phi_1^s,\Phi_2^s$, $\Phi_3^s $ and $\Phi_4^s $:
\begin{align*}
\Phi_j^s(\z_1,\dots,\z_{M(j)})=\frac{1}{(M(j))!} \sum_{\tau \in \mathcal S_{M(j)}} \Phi_j(\z_{\tau(1)},\dots,\z_{\tau(M(j))})
\end{align*}
for $j=1,\dots,4$ where $\mathcal S_{k}$ is the symmetric group of order $k$ (that is the set of all permutations on $I_k$).
For $j=1,\dots 4$, the integrals $I(\Phi_j^s)$ are naturally estimated by U-statistics of order $M(j)$. More precisely, we consider an i.i.d.\ sample $\left(\ZZ_1,\dots,\ZZ_N\right)$ ($N\geqslant 1$) with distribution $\P_2^{\textbf{u}}$ and, 
  for $j=1,\dots,4$, we define the U-statistics
\begin{align}\label{def:estU}
U_{j,N}&= \begin{pmatrix}N\\M(j)\end{pmatrix}^{-1}\sum_{1\leq i_1<\dots<i_{M(j)}\leq N}\Phi_j^s\left(
\ZZ_{i_1},\dots,\ZZ_{i_{M(j)}}
\right).
\end{align}
Theorem 7.1 in \cite{Hoeffding48} ensures that $U_{j,N}$ converges in probability to $I(\Phi_j)$ for any $j=1,\dots,4$. Moreover, one may also prove that the convergence holds almost surely proceeding as in the proof of Lemma 6.1 in \cite{GKL18}. 
Then we estimate $S_{2,GMS}^{\textbf{u}}$ by 
\begin{equation}\label{def:est}
\widehat{S}_{2,GMS}^{\textbf{u}}= \frac{U_{1,N}-U_{2,N}}{U_{3,N}-U_{4,N}}=\Psi(U_{1,N},U_{2,N},U_{3,N},U_{4,N}).
\end{equation}

\begin{rmk}
Naturally, a covariance quantity $\Cov(A,B)$ can be estimated using either the expression $\Cov(A,B)=\E[AB]-\E[A]\E[B]$ or 
the expression $\Cov(A,B)=\E[(A-\E[A])(B-\E[B])]$
leading to the following estimators:
\[
\frac 1N \sum_{i=1}^N A_iB_i - \left(\frac 1N \sum_{i=1}^N A_i\right)\left( \frac 1N \sum_{i=1}^N B_i\right) \quad \text{or} \quad \frac 1N \sum_{i=1}^N \left(A_i-\frac 1N \sum_{i=1}^N A_i\right)\left(B_i-\frac 1N \sum_{i=1}^N B_i\right)
\]
which are equal. 
The use of the right-hand side formula enables greater numerical stability (i.e., less error due to round-offs). The Kahan compensated summation algorithm \cite{kahan1965pracniques} may also be used on these sums. However, the left-hand side formula is generally preferred in sensitivity analysis for the mathematical analysis. This analysis is of course independent of the way the estimators are numerically computed in practice. The same holds for a variance term. 

Hence, the estimator $\widehat{S}_{2,GMS}^{\textbf{u}}$ defined in \eqref{def:est} can be rewritten in the fashion of the right-hand side of the previous display.  
\end{rmk}

Our main result follows.

\begin{theorem}\label{th:clt}
If for $j=1,\dots,4$, $\E\left[\Phi_j^s\left(\ZZ_1,\dots,\ZZ_{M(j)}\right)^2\right]<\infty$ then
\begin{align}\label{eq:clt}
\sqrt{N}\left(\widehat{S}_{2,GMS}^{\textbf{u}}-S_{2,GMS}^{\textbf{u}}\right)\cvloi\mathcal{N}_1(0,\sigma^2)
\end{align}
where the asymptotic variance $\sigma^2$ is given by \eqref{def:sigma} in the proof below.
\end{theorem}

\begin{proof}[Proof of Theorem \ref{th:clt}]
The first step of the proof is to apply Theorem 7.1 of \cite{Hoeffding48} to the random vector $\left(U_{1,N},U_{2,N},U_{3,N},U_{4,N}\right)^{\top}$. By Theorem 7.1 and Equations (6.1)-(6.3) in \cite{Hoeffding48}, it follows that 
\[
\sqrt{N}\left(
\begin{pmatrix}U_{1,N}\\U_{2,N}\\U_{3,N}\\U_{4,N}\end{pmatrix}
-
\begin{pmatrix}I(\Phi_1^s)\\I(\Phi_2^s)\\I(\Phi_3^s)\\I(\Phi_4^s)\end{pmatrix}
\right)
\cvloi \mathcal{N}_4(0,\Gamma)
\] 
where $\Gamma$ is the square matrix of size $4$ given by 
\[
\Gamma(i,j)= M(i)M(j) \Cov(\E[\Phi_i^s(\ZZ_1,\dots,\ZZ_{M(i)})\vert \ZZ_1],\E[\Phi_j^s(\ZZ_1,\dots,\ZZ_{M(j)})\vert \ZZ_1]).
\]
Now, it remains to apply the so-called delta method (see \cite{van2000asymptotic}) with the function $\Psi$ defined by \eqref{def:psi}. Thus, one gets the asymptotic behavior in Theorem \ref{th:clt} where $\sigma^2$ is given by
\begin{align}\label{def:sigma}
\sigma^2= g^{\top} \Gamma g
\end{align}
with $g=\nabla \Psi(I(\Phi_1^s),I(\Phi_2^s),I(\Phi_3^s),I(\Phi_4^s))$ and $\nabla \Psi = (z-t)^{-2}\left( z-t, -z+t, -x+y, x-y\right)^{\top}$.
\end{proof}

Notice that we consider $(m+2)$ copies of $\ZZ$ in the definition of $S_{2,GMS}^{\textbf{u}}$ (see \eqref{eq:USTAT2}). Nevertheless, the estimation procedure only requires a $N$ sample of $\ZZ$ (see \eqref{def:est}) that means only $2N$ evaluations of the black-box code which constitutes an appealing advantage of the method presented in this paper. Moreover, the required number of calls to the black-box code is independent of the size $m$ of the class of tests functions unlike in \cite{GKL18} or in \cite{FGM2017} where $(m+2)\times N$ calls to the computer code were necessary. In addition,  the proof of the asymptotic normality in Theorem \ref{th:clt} is  elementary and does not rely anymore on the use of the sophisticated functional delta method as in \cite{GKL18}.

\subsection{Comments}\label{ssec:comments}

For any output code $f$, one may consider 
different choices of the family $(T_a)_{a\in \mathcal X^m}$ of functions indexed by $a\in \mathcal X^m$ leading to very different indices. The choice of the family must be induced by the aim of the practitioner. To quantify the output sensitivity around the mean, one should consider the classical Sobol indices based on the variance and corresponding to the first particular case presented in Section \ref{ssec:index}. Otherwise, interested in the sensitivity of the whole distribution, one should prefer a family of functions that characterizes the distribution.
For instance, in the second particular case presented in Section \ref{ssec:index}, the functions $T_a$ are the indicator functions of half-lines and yield the Cramér-von-Mises indices.

Moreover, since in the estimation procedure the number of output calls is independent of the choice of the family $(T_a)_{a\in \mathcal X^m}$, one can consider and estimate simultaneously several indices with no-extra cost. In fact, the only computational challenge relies in our capability to evaluate the functions $\Phi$ on the sample.

\section{Applications in classical frameworks and beyond}\label{sec:appli}

\subsection{Particular cases}

\paragraph{Sobol indices}

In the case where  $\mathcal X=\R$, $m=0$ and the test functions $T_a$ given by $T_a(x)=x$ (do not depend on the parameter $a$), we recover the  classical Sobol indices. As mentioned in the Introduction, many classical methods of estimation are available. Among them, one can cite estimation procedure based on polynomial chaos expansion \cite{Sudret2008global}, quasi Monte-Carlo scheme \cite{Kucherenko2017different,owen2},  the classical  Pick-Freeze method \cite{pickfreeze,janon2012asymptotic}, and more recently a method based on nested Monte-Carlo \cite{GODA201763}. This last method seems to be numerically efficient. Nevertheless, it requires that all the random elements have a density with respect to the Lebesgue measure to be able to simulate under the conditional distribution. In addition, no theoretical asymptotic convergences are given.

\medskip

As explained in Section \ref{ssec:est}, our method provides 
a new estimator based on U-statistics for the classical Sobol index. In that case,
the estimator is given by \eqref{def:est} and, for $j=1,\ldots,4$, the $U_{j,N}$'s are given by
\begin{align*}
U_{1,N}&= \frac 1N \sum_{i=1}^N Z_iZ_i^{\textbf{u}}\\
U_{2,N}&= \frac{1}{N(N-1)}\left( \sum_{i=1}^N Z_i\sum_{i=1}^N Z_i^{\textbf{u}}-\sum_{i=1}^N Z_iZ_i^{\textbf{u}}\right)\eqdef \frac{1}{N(N-1)}(\tilde U_{2,N}-\tilde V_{2,N})\\
U_{3,N}&= \frac 1N \sum_{i=1}^N Z_i^2\\
U_{4,N}&= \frac{1}{N(N-1)}\left( \left(\sum_{i=1}^N Z_i\right)^2-\sum_{i=1}^N Z_i^2\right)\eqdef \frac{1}{N(N-1)}(\tilde U_{4,N}-\tilde V_{4,N})
\end{align*}
leading to
\begin{equation*}
\widehat S^{\textbf{u}}_{2,GMS}= \frac{U_{1,N}-U_{2,N}}{U_{3,N}-U_{4,N}}=\Psi(U_{1,N},U_{2,N},U_{3,N},U_{4,N})
\end{equation*}
while in \cite{pickfreeze}, the classical Pick-Freeze estimator $\widehat S^{\textbf{u}}$ of $S^{\textbf{u}}_{2,GMS}$ is given by 
\begin{equation}\label{def:est_R}
\widehat S^{\textbf{u}}= \frac{U_{1,N}-(1-1/N^2) \tilde U_{2,N}}{U_{3,N}-(1-1/N^2) \tilde U_{4,N}}=\Psi(U_{1,N},(1-1/N^2) \tilde U_{2,N},U_{3,N},(1-1/N^2) \tilde U_{4,N})
\end{equation}
and takes into account the diagonal terms. Both procedures require $2N$ evaluations of the black-box code and have the same rate of convergence. The estimators are slightly different which induces different asymptotic variances.
Finally, one may improve the estimation $\widehat S^{\textbf{u}}$ using the information of the whole sample leading to $\widehat T^{\textbf{u}}$ given in \cite[Equation (6)]{pickfreeze}:
\begin{align}\label{def:Tn}
\widehat T^{\textbf{u}}=\frac{
\frac{1}{N}\sum_{i=1}^NY_iY_i^{\textbf{u}}-\left(\frac{1}{N}\sum_{j=1}^N\frac{Y_i+Y_i^{\textbf{u}}}{2}\right)^2}
{
\frac{1}{N}\sum_{i=1}^N\frac{(Y_i)^2+(Y_i^{\textbf{u}})^2}{2}-\left(\frac{1}{N}\sum_{i=1}^N\frac{Y_i+Y_i^{\textbf{u}}}{2}\right)^2
}.
\end{align}
The sequence of estimators $\widehat T^{\textbf{u}}$ is asymptotically efficient in the Cramér-Rao sense (see \cite[Proposition 2.5]{pickfreeze}). In this paper, we also could have constructed a new estimator $\widehat T^{\textbf{u}}_{2,GMS}$ analog version of $\widehat S^{\textbf{u}}_{2,GMS}$ taking into account the whole information contained in the sample. However, based on the same initial design as $\widehat S^{\textbf{u}}$ and $\widehat T^{\textbf{u}}$, neither $\widehat S^{\textbf{u}}_{2,GMS}$ nor $\widehat T^{\textbf{u}}_{2,GMS}$ will be asymptotically efficient. Nevertheless, the estimation procedure proposed in this paper outperforms the procedure presented in \cite{GKL18,FGM2017} as soon as $m\geqslant 1$.

\paragraph{Sobol indices for multivariate outputs} For $\mathcal X=\R^k$ and $m=0$, one may realize the same analogy between the estimation procedure proposed in this paper and that in \cite{GJKL14}.

\paragraph{Cramér-von-Mises indices} For $\mathcal X=\R^k$, $m=1$ and the test functions $T_a$ given by $T_a(x)=\ind_{\{x\leqslant a\}}$, we outperform the central limit theorem proved in \cite{GKL18}. Indeed, the estimator proposed in \cite{GKL18} requires $3N$ evaluations of the computer code versus only $2N$ in our new procedure. In addition, the proof therein is based on the powerful but complex functional delta method while the proof of Theorem \ref{th:clt} is an elementary application of Theorem 7.1 in \cite{Hoeffding48} combined with the classical delta method.

\subsection{Compact manifolds}

A particular framework is the case when the output space  is a compact Riemannian manifold $\mathcal{M}$. In \cite{FGM2017}, a similar index to $S_{2,GMS}^{\textbf{u}}$ is studied in this special context, taking $T_a(x)=\ind_{\{x\in \widetilde B(a_1,a_2)\}}$ as test functions, where $\widetilde B(a_1,a_2)$ still stands for the ball whose center is the middle point between $a_1$ and $a_2$ with radius $\overline{a_1a_2}/2$.
The authors showed that, under some restrictions on the underlying probability measure and the Riemannian manifold, the  family of balls $\bigl( \widetilde B(a_1,a_2) \bigr)_{(a_1,a_2) \in \mathcal{M}}$ is a determining class, that is, if two probability measures $\P_1$ and $\P_2$ on $\mathcal{M}$  coincide on all the events of this family, then $\P_1=\P_2$. By this property, they proved that if their index, denoted $B_2^{\textbf{u}}$, vanishes 
 then the distributions of $T_a(Z)$ and  $(T_a(Z)|X_{\textbf{u}})$ coincide. Further, the performance of $B_2^{\textbf{u}}$ in Riemannian manifolds immersed in $\mathbb{R}^d$ with $d = 2,3$ and on the cone of positive definite matrices (manifold) is analyzed. 
Last, an exponential inequality for the  estimator $\hat{B}_2^{\textbf{u}}$ of ${B}_2^{\textbf{u}}$ is provided together with the almost sure convergence  that is deduced from. Unfortunately, no central limit theorem is given. 

As a particular case of $S_{2,GMS}^{\textbf{u}}$, the asymptotic distribution of $\hat{B}_2^{\textbf{u}}$ can be found from Theorem \ref{th:clt}.  Given $x$, since $(a_1,a_2)\mapsto T_{(a_1,a_2)}(x)$  is a symmetric function and  $m=2$, it is verified that,
\begin{align*}
& \Phi_1(\z_1,\z_2,\z_3)= \ind_{ \left \{ \z_3, \z_{3}^{\textbf{u}}  \in \widetilde  B(\z_1,\z_2) \right \}},\\ 
& \Phi_2(\z_1,\z_2,\z_3,\z_4)= \ind_{\left \{\z_3, \z_{4}^{\textbf{u}}  \in \widetilde B(\z_1,\z_2) \right \}}, \\
& \Phi_3(\z_1,\z_2,\z_3)=  \ind_{\left \{\z_3 \in \widetilde B(\z_1,\z_2) \right \} } ,\\
& \Phi_4(\z_1,\z_2,\z_3,\z_4)=  \ind_{\left \{\z_3, \z_{4}  \in \widetilde B(\z_1,\z_2) \right \}}. 
\end{align*}

In this setting, the limiting covariance matrix $\Gamma$ is given by $\Gamma(i,j)= M(i)M(j)\Cov \left( L_i, L_j \right)$, for $i,j=1,\ldots,4$ where
\begin{align*}
L_1&= \frac{1}{6} \sum_{\tau \in \mathcal S_{3}}  \mathbb{P}  \left ( Z_{\tau_3}, Z_{\tau_3}^{\textbf{u}}  \in  \widetilde B(Z_{\tau_1},Z_{\tau_2}) \vert Z_1 \right ), \\ 
L_2&= \frac{1}{24} \sum_{\tau \in \mathcal S_{4}}  \mathbb{P}  \left (   Z_{\tau_3}, Z_{\tau_4}^{\textbf{u}}  \in \widetilde B(Z_{\tau_1},Z_{\tau_2}) \vert  Z_1 \right ), \\
L_3 &= \frac{1}{6} \sum_{\tau \in \mathcal S_{3}}  \mathbb{P} \left ( Z_{\tau_3}  \in \widetilde B(Z_{\tau_1},Z_{\tau_2}) \vert Z_1   ) \right ),\\
L_4&= \frac{1}{24} \sum_{\tau \in \mathcal S_{4}}  \mathbb{P} \left (   Z_{\tau_3}, Z_{\tau_4}  \in \widetilde B(Z_{\tau_1},Z_{\tau_2}) \vert  Z_1 \right ).
\end{align*}


\section{Numerical applications}\label{sec:num}

%
%
%
%
%

\subsection{A non linear model}

In this section, we illustrate and we compare the different estimation procedures based on the Pick-Freeze scheme and the U-statistics for the classical Sobol indices 
on the following toy model:
\begin{equation}\label{eq:exp1}
Z=\exp\{ X_1+ 2X_2\}, 
\end{equation}
where $X_1$ and $X_2$ are independent standard Gaussian random variables. The distribution of $Z$ is log-normal and we can derive both its probability density function and its c.d.f.\:
\[
f_Z(z)=\frac{1}{\sqrt{10\pi}z}e^{-(\ln z)^2/10}\ind_{\R_+}(z)\quad \trm{and} \quad F_Z(z)=\Phi\left(\frac{\ln z}{\sqrt 5}\right),
\]
where $\Phi$ stands for the c.d.f.\ of the standard Gaussian random variable. 
We have $p=2$ input variables and tedious exact computations  (see \cite{GKL18}) lead to closed forms of the Sobol indices:
\begin{align*}
S^1=\frac{1-e^{-1}}{e^4-1}\approx 0.0118     \quad \text{and} \quad S^2=\frac{e^3-e^{-3}}{e^4-1}\approx 0.3738.
\end{align*}
Further, the Cramér-von-Mises indices $S_{2,CVM}^{1}$ and $S_{2,CVM}^{2}$ are also explicitly computable:
\begin{align*}
S_{2,CVM}^{1}=\frac{6}{\pi}\arctan 2-2\approx 0.1145 \quad \text{and} \quad
S_{2,CVM}^2=\frac{6}{\pi}\arctan \sqrt{19}-2\approx 0.5693. 
\end{align*}
The reader is refered to \cite{GKL18} for the details of these computations. 

In Figure \ref{fig:jouet_N_increases}, we compare the estimations of the two first-order Sobol indices 
obtained by both estimation procedures (U-statistics and Pick-Freeze). The total number of calls of the computer code ranges from $n=100$ to $500000$. When estimating the Sobol indices with both methodologies, 
we have considered samples of size $N=n/(p+1)$ so that each estimation requires a total number $n$ of evaluations of the code.   
Analogously, when estimating the Cramér-von-Mises indices using U-statistics, we have also considered samples of size $N=n/(p+1)$.  
In contrast, 
when estimating the Cramér-von-Mises indices using the Pick-Freeze scheme, we have considered 
samples of size $N=n/(p+2)$. 
 We observe that both methods converge and give precise results for large sample sizes. The same kind of convergence can be observed for the estimations of the Cramér-von-Mises indices with both methodologies. Actually, the convergence is a bit slower which is not surprising due to the greater complexity of the Cramér-von-Mises indices. 
In addition, the estimation procedure with U-statistics outperforms the Pick-Freeze one as soon as $m\geqslant 1$. As already mentioned in Sections \ref{ssec:est} and \ref{ssec:comments}, such a better performance increases with the number $m$ of parameters of the tests functions family. Indeed, for a fixed budget $n$ (in other words, a fixed number of evaluations of the computer code), the needed sample size to estimate the $p$ first-order indices with the standard Pick-Freeze scheme is $N_1=n/(m+p+1)$ to be compared to $N_2=n/(p+2)$ required in the $U$-statistics estimation procedure.  

\begin{figure}[h!]
\centering
\begin{tabular}{l}
\includegraphics[scale=.92]{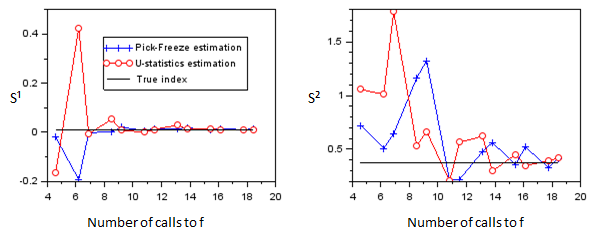}
\end{tabular}
\caption{Non-linear model \eqref{eq:exp1}. Convergence of both methods when the total number of calls of the computer code increases. The two first-order Sobol indices have been represented from left  to right. Several total number of calls of the computer code have been considered ranging from $n=100$ to $n=10^8$. The $x$-axis is in logarithmic scale.
}
\label{fig:jouet_N_increases}
\end{figure}

\subsection{The Gaussian plume model}\label{ssec:plume}

In this section, the model under study concerns  a point source that emits contaminant into a uni-directional wind in an infinite domain. Such a model is also applied, for instance, to volcanic eruptions, pollen and insect dispersals, and is called the Gaussian plume model (GPM) (see, e.g., \cite{carrascal1993,stockie2011}). The GPM assumes that atmospheric turbulence is stationary and homogeneous.  
Naturally, in Earth Sciences, it is crucial to analyze the sensitivity of the output of the GPM model regarding the input parameters (see \cite{mahanta2012,pouget2016}).

The model parameters are represented in Figure \ref{pluma}. The contaminant is emitted at a constant rate $Q$ and the wind direction is denoted by $\textbf{u}=(u,0,0)$ (with $u \geq 0)$ while the effective height is $H=h(1+\delta)$ where $h$ is the stack height and $\delta h$ is the plume rise.

\begin{figure}[ht]\centering 
\includegraphics[width=.49\textwidth]{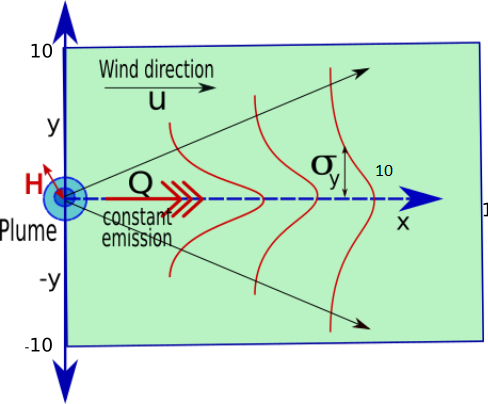}
\caption{Plume model. Cross section at $z=0$ of a contaminant plume emitted from a continuous point source, with wind direction aligned with the $x$–axis.}
\label{pluma}
\end{figure}

Then the contaminant concentration at location $(x,y,z)$ is given by
\[
C(x,y,z)= \frac{Q}{4 \pi u r(x)} e^{-\frac{y^2}{4r(x)}}  \left(e^{-\frac{(z-H)^2}{4r(x)}} + e^{-\frac{(z+H)^2}{4r(x)}}  \right),
\]
where $r$ is a  parametric function given by $r(x)= \frac{1}{u}\int_0^x K(v)dv$, the function $K$ being the eddy diffusion. 
In this section, we investigate the particular two-dimensional case: the height is considered as zero (at ground level). 
In addition, we suppose that $r(x)=Kx/u$ where $K$ is a constant. Hence, the contaminant concentration at location $(x,y,0)$ rewrites as: 
\begin{align}\label{def:plume}
C(x,y,0)= \frac{Q}{2\pi K x} e^{\frac{-u(y^2 +H^2)}{4Kx}}.
\end{align}

A first step consists in performing a GSA for spatial data, namely an ubiquous sensitivity analysis. In other words, the sensitivity indices are computed location after location leading to a sensitivity map. See, for instance,   \cite{marrel2017sensitivity} for more details on this methodology. The results are presented in Figure \ref{fig:plume}.

\begin{figure}[h!]
\centering
\begin{tabular}{l}
\includegraphics[scale=1.2]{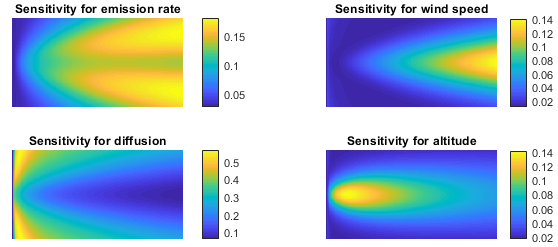}
\end{tabular}
\caption{Plume model \eqref{def:plume}. Ubiquous sensitivity analysis with respect to the emission rate $Q$ (top left), the wind speed $u$ (top right), the diffusion $K$ (bottom left) and the altitude $H$ (bottom right). 
}
\label{fig:plume}
\end{figure}

Secondly,  we wish to perform a sensitivity analysis globally on the contaminant concentration with respect to the uncertain inputs $Q$, $K$, and $u$, while the altitude plume parameter $H$ is fixed in advance. In this setting, the function $f$ that defines the output of interest in \eqref{def:model} is then given by:
\begin{align}\label{def:plume_as}
 \begin{matrix}
f\colon &\R^3 &\to &L^2(\R^2)\\
&(Q,K,u)&\mapsto &f(Q,K,u)=\left(C(x,y,0)\right)_{(x,y)\in \R^2}.
\end{matrix}
\end{align}
In other words, to any triplet $(Q,K,u)$, the computer code associates one square-integrable field from $\R^2$ to $\R$.  
 Based on reality constraints and guided by the expert knowledge, the stochastic  parameters $Q$, $K$, and $u$ of the model are assumed to be all independent with uniform distribution on $[0,10]$. Let $C_1$ and $C_2$ be two pollution concentrations  with domain in the ground level. The range of values of $x$ and $y$ where $C_1$ and $C_2$ are compared is $A=\{ (x,y) \in \mathbb{R}^2 : x \in [0,10], y \in [-10,10] \}$. The distance  used  is the classical $L^2$ distance 
\[
d(C_1,C_2)=\sqrt{\iint (C_1(x,y,0)-C_2(x,y,0))^2 dxdy}.
\] 

To quantify the sensitivity on the contaminant concentration with respect to $Q$, $K$, and $u$, we consider the family of test functions $T_a$ given by $T_{(a_1,a_2)}(b)= \ind_{b \in B(a_1,a_2)}$, where $a_1$, $a_2$, and $b$ are square-integrable applications from $\R^2$ to $\R$ and $B(a_1,a_2)$ stands for the ball centered at $a_1$ with radius $\overline{a_1a_2}$ (whence $m = 2$). The values of the indices are presented in Table  \ref{power_norm}. In this study, we have considered several values of the altitude plume parameter $H$ from 1 to 20 and a sample size $N$ equal to 1000, 2000, and 5000. 
We observe that,  as $H$  increases, the values of the sensitivity indices decrease. When $N=5000$, we may also observe that the rank of the indices largely varies with respect to the value of $H$: for large values of $H$, the parameter $K$ appears to be the most influent 
on the concentration. In contrast, when $H=1$, all three parameters seem to have the same influence.

%

  \begin{table}[ht]
\begin{center}
\scalebox{0.90}{
\begin{tabular}{l | ccc | ccc| ccc}

\multicolumn{1}{c}{}  &  \multicolumn{3}{c}{{N=1000}} & \multicolumn{3}{|c}{{N=2000}}  & \multicolumn{3}{|c}{{N=5000}}          \\
\cline{2-10}
\multicolumn{1}{c}{} &  {K}  & {Q }& {u} & {K} & {Q} &  {u} & {K} & {Q} & {u}                \\
\hline
  {H=1}  & 0.1365    & 0.1216  & 0.1330  & 0.1124 	        & 0.1419 	 	    &0.1453 & 0.1425	& 0.1431		& 0.1562 \\
  {H=2}  & 0.1028    &0.1197    & 0.1212    & 0.1291 	    & 0.1317	        &0.1171 		& 0.1222	        & 0.1627		& 0.1143\\
  {H=10} & 0.0813    & 0.0891 & 0.1010  &0.1081       	& 0.1077 		    &0.1256  	& 0.0893 	& 0.0831 	& 0.1001\\
{H=20}   & 0.1027     & 0.0246    & 0.1041    & 0.0620 	& 0.0942	 	&0.1030  	& 0.0913 	         & 0.0091 		& 0.0329\\
\end{tabular}}

\caption{Sensitivity indices for the plume model \eqref{def:plume}}

\label{power_norm}
\end{center}
\end{table}

\subsection{Singular value decomposition in partial differential equation}\label{ssec:svd}

In this section, we study the sensitivity of the solution (numerical approximation) of a partial differential equation, when the parameters of the equation (inputs) vary. In particular, we analyze the sensitivity of the subspaces generated by the singular value decomposition  of the numerical grid output matrix solution. Many problems can be modelled by an elliptical differential partial equation. For instance, in physics, electric potential, potential flow, structural mechanics are all studied, see \cite{sobolev2016}. In biology, the reaction–diffusion–advection equation is used to model chemotaxis observed in bacteria, population migration and  evolutionary adaptation to changing environments, see \cite{volpert2011}. 

In this setting, it is usual to compact the information through the singular value decomposition of the solution matrix, that is, the numerical solution of the differential equation.  Furthermore, it can also be useful to analyze the influence of the parameters in this information compactification.

In this section,  an elliptical  differential partial equation of type diffusive transport is considered:
\begin{align}\label{eq:ex2}
B \frac{\partial C}{\partial t}= \frac{\partial}{\partial x} \left[D \frac{\partial C}{\partial x}\right] + \frac{\partial}{\partial y} \left[D \frac{\partial C}{\partial y}\right]  -rC +p_{xy}, 
\end{align}
with production rate $p_{xy}$ at location $(x,y)$, consumption rate $r$, and diffusive transport $D$ of a substance $C$ in time $t$ and spatial dimensions $(x, y)$. The boundaries are prescribed as zero-gradient (default value). The parameter $p_{xy}$ is zero everywhere except in $50$ randomly positioned spots denoted by $(x_i,y_i)$, for $i=1,\ldots, 50$.
We assume that  the production rate is the same at any of the 50 locations and equal to $p$ and we consider that the function $f$ in \eqref{def:model} defining the output $C$ is given by 
\begin{align}\label{def:svd_as}
 \begin{matrix}
f\colon &\R^4 &\to &L^2(\R_+\times \R^2)\\
&(B,D,r,p)&\mapsto &f(B,D,r,p)=\left(C(t,x,y)\right)_{(t,x,y)\in \R_+\times \R^2}.
\end{matrix}
\end{align}
All the input parameters are then assumed to be uniformly distributed:
\begin{align*}
B &\sim \mathcal U([1-\beta,1+\beta]), \\
D &\sim \mathcal U([2-\delta,2+\delta]), \\
r &\sim \gamma \cdot{}  \mathcal U([1,2]), \\
p &\sim \mathcal U([0,1]).
\end{align*}
Let $C(0,x,y)$ be the solution of \eqref{eq:ex2} at time $t=0$. We compute the matrix $A$ that has the first two principal component scores along its columns. Note that these two columns represent a rank-2 approximation of the matrix solution. This matrix $A$ is a way to embed  the approximated solution on a Stiefel manifold $S_t$. That is,  $A \in S_t= \{ M \in\mathcal M_{50,2}: M^\top M=Id \}$, where $\mathcal M_{n,k}$ stands for the set of matrices of size $n\times k$. We consider the Stiefel manifold as an embedded one into the Euclidean space, and we choose the standard inner product in this space (Frobenius product) as metric in the Riemannian manifold. Notice that it is also possible to select another metric in $S_t$. Therefore,
the similarity between two matrices is given by the Frobenius distance, that is, for any matrices $A_1$ and $B_2 \in \mathcal M_{n,k}$, 
\[
d(A_1,A_2)= \sqrt{ \tr ((A_1-A_2)^\top (A_1-A_2))},
\]
 where $\tr(A)$ represents the trace of the matrix $A$.
We consider the parametric family of functions given by
\[
T_{(A_1,A_2)}(\cdot{} )= \ind_{\cdot{} \in B(A_1,A_2) \cap B(A_2, A_1)},
\]
where the parameters $A_1$ and $A_2$ of the test functions are now matrices (thus are written with capital letters) and $B(A_1, A_2)$ (resp. $B(A_2, A_1)$) still stands for the ball centered at $A_1$ (resp. $A_2$)
with radius $\overline{A_1A_2}$.
In Table \ref{power_norm2}, the sensitivity indices are calculated for different values of $\beta$, $\delta$, and $\gamma$ and the high influence of the parameter $r$ is observed in all cases. As expected, this influence increases with $\gamma$  and decreases as the value of $\delta$ increases.  The simulations have been generated using the R language \cite{r2019}. In particular, the discretized solution of the differential equation has been computed with the \texttt{ReacTran} package \cite{soetaert2012}. 

\begin{table}[ht]
\begin{center}
\scalebox{0.95}{
\begin{tabular}{c | ccc | ccc}
  &  \multicolumn{3}{c|}{{$\delta=0.1$}} & \multicolumn{3}{c}{{$\delta=0.5$}}            \\
\hline
{$\gamma=0.001$} &  { $B$}  & { $D$ }& { $r$}  &  {$B$}  & {$D$}& {$r$}   \\
  \hline
  {$\beta=0.1$}  &  0.001  & 0.011  & 0.546   	        & 0.020	 	    &0.071 	& 	0.119 	    	 \\
  {$\beta=0.5$}   & 0.010   & 0.007    &  0.491    	    & 0.000	        &0.041 	 & 0.102 	        \\
  \hline
  \hline
&  \multicolumn{3}{c|}{{$\delta=0.1$}} & \multicolumn{3}{c}{{$\delta=0.5$}}            \\
\hline
{$\gamma=0.01$}  &  {$B$}  & {$D$}& {$r$}  &  {$B$}  & {$D$}& {$r$}   \\
\hline
  {$\beta=0.1$}  & 0.000   & 0.001    & 0.664      	        & 0.010 	 	    &0.053 	& 0.168 	 	    		 \\
  {$\beta=0.5$}   & 0.013   &0.006    & 0.621    	    & 0.008	        &0.041 	& 0.132 	 	   		\\
\hline
\hline
 &  \multicolumn{3}{c|}{{$\delta=0.1$}} & \multicolumn{3}{c}{{$\delta=0.5$}}            \\
\hline
{$\gamma=0.1$} &  {$B$}  & {$D$}& {$r$}  &  {$B$}  & {$D$}& {$r$}    \\
\hline
  {$\beta=0.1$}  & 0.005   &0.005    & 0.794     	        & 0.020 	 	    &0.051 	& 0.179  	 \\
  {$\beta=0.5$}   & 0.000   &0.006    & 0.721   	    & 0.000	        &0.043 	& 0.171 	 	   \\

\end{tabular}}
\caption{Sensitivity indices for the partial differential equation \eqref{eq:ex2}}

\label{power_norm2}
\end{center}
\end{table}

\section{Conclusion}\label{sec:concl}

In this paper, we explain how to construct a large variety of sensibility indices when the output space of the black-box model is a general metric space. This construction encompasses the classical Sobol indices  \cite{janon2012asymptotic} and their vectorial generalization \cite{GJKL14} as well as some indices based on the whole distribution, namely the Cramér-von-Mises indices \cite{GKL18}. In addition, we propose an estimation procedure  that ensures strong consistency and asymptotic normality at a cost of $2N$ calls to the computer code with a rate of convergence  $\sqrt N$.
As soon as $m\geqslant 1$, this new methodology appears to be more efficient than the so-called Pick-Freeze estimation procedure.

\medskip

\textbf{Acknowledgment}. We warmly thank Anthony Nouy and Bertrand Iooss for the numerical examples of Sections \ref{ssec:plume} and \ref{ssec:svd}. Support from the ANR-3IA Artificial and Natural Intelligence Toulouse Institute is gratefully acknowledged.
The authors are indebted to the anonymous reviewers for their helpful comments
and suggestions, that lead to an improvement of the manuscript. In particular, we warmly thank one of the reviewers for pointing us the interest of performing an ubiquous sensitivity analysis in Section \ref{ssec:plume}.

%
%
%
%


 \bibliographystyle{abbrv}
\bibliography{biblio_GMS}

\def\cprime{$'$}
\begin{thebibliography}{10}

\bibitem{borgonovo2007}
E.~Borgonovo.
\newblock A new uncertainty importance measure.
\newblock {\em Reliability Engineering \& System Safety}, 92(6):771--784, 2007.

\bibitem{borgonovo2011moment}
E.~Borgonovo, W.~Castaings, and S.~Tarantola.
\newblock Moment independent importance measures: New results and analytical
  test cases.
\newblock {\em Risk Analysis}, 31(3):404--428, 2011.

\bibitem{borgonovo289probabilistic}
E.~Borgonovo, G.~B. Hazen, V.~R.~R. Jose, and E.~Plischke.
\newblock Probabilistic sensitivity measures as information value.
\newblock {\em European Journal of Operational Research}, 289(2):595--610.

\bibitem{borgonovo2016common}
E.~Borgonovo, G.~B. Hazen, and E.~Plischke.
\newblock A common rationale for global sensitivity measures and their
  estimation.
\newblock {\em Risk Analysis}, 36(10):1871--1895, 2016.

\bibitem{carrascal1993}
M.~Carrascal, M.~Puigcerver, and P.~Puig.
\newblock Sensitivity of gaussian plume model to dispersion specifications.
\newblock {\em Theoretical and Applied Climatology}, 48(2-3):147--157, 1993.

\bibitem{DaVeiga13}
S.~Da~Veiga.
\newblock Global sensitivity analysis with dependence measures.
\newblock {\em J. Stat. Comput. Simul.}, 85(7):1283--1305, 2015.

\bibitem{rocquigny2008uncertainty}
E.~De~Rocquigny, N.~Devictor, and S.~Tarantola.
\newblock {\em Uncertainty in industrial practice}.
\newblock Wiley Chisterter England, 2008.

\bibitem{FKR13}
J.-C. Fort, T.~Klein, and N.~Rachdi.
\newblock New sensitivity analysis subordinated to a contrast.
\newblock {\em Comm. Statist. Theory Methods}, 45(15):4349--4364, 2016.

\bibitem{FGM2017}
R.~Fraiman, F.~Gamboa, and L.~Moreno.
\newblock Sensitivity indices for output on a riemannian manifold.
\newblock {\em International Journal for Uncertainty Quantification}, 10(4),
  2020.

\bibitem{GJKL14}
F.~Gamboa, A.~Janon, T.~Klein, and A.~Lagnoux.
\newblock Sensitivity analysis for multidimensional and functional outputs.
\newblock {\em Electronic Journal of Statistics}, 8:575--603, 2014.

\bibitem{pickfreeze}
F.~Gamboa, A.~Janon, T.~Klein, A.~Lagnoux, and C.~Prieur.
\newblock Statistical inference for {S}obol pick-freeze {M}onte {C}arlo method.
\newblock {\em Statistics}, 50(4):881--902, 2016.

\bibitem{GKL18}
F.~Gamboa, T.~Klein, and A.~Lagnoux.
\newblock Sensitivity analysis based on {C}ram\'er--von {M}ises distance.
\newblock {\em SIAM/ASA J. Uncertain. Quantif.}, 6(2):522--548, 2018.

\bibitem{GODA201763}
T.~Goda.
\newblock Computing the variance of a conditional expectation via non-nested
  {M}onte {C}arlo.
\newblock {\em Operations Research Letters}, 45(1):63 -- 67, 2017.

\bibitem{Hoeffding48}
W.~Hoeffding.
\newblock A class of statistics with asymptotically normal distribution.
\newblock {\em Ann. Math. Statistics}, 19:293--325, 1948.

\bibitem{janon2012asymptotic}
A.~Janon, T.~Klein, A.~Lagnoux, M.~Nodet, and C.~Prieur.
\newblock Asymptotic normality and efficiency of two {S}obol index estimators.
\newblock {\em ESAIM: Probability and Statistics}, 18:342--364, 1 2014.

\bibitem{kahan1965pracniques}
W.~Kahan.
\newblock Pracniques: further remarks on reducing truncation errors.
\newblock {\em Communications of the ACM}, 8(1):40, 1965.

\bibitem{kala2019quantile}
Z.~Kala.
\newblock Quantile-oriented global sensitivity analysis of design resistance.
\newblock {\em Journal of Civil Engineering and Management}, 25(4):297--305,
  2019.

\bibitem{Kucherenko2017different}
S.~Kucherenko and S.~Song.
\newblock Different numerical estimators for main effect global sensitivity
  indices.
\newblock {\em Reliability Engineering \& System Safety}, 165:222--238, 2017.

\bibitem{lamboni2011multivariate}
M.~Lamboni, H.~Monod, and D.~Makowski.
\newblock Multivariate sensitivity analysis to measure global contribution of
  input factors in dynamic models.
\newblock {\em Reliability Engineering \& System Safety}, 96(4):450--459, 2011.

\bibitem{luyi2012moment}
L.~Luyi, L.~Zhenzhou, F.~Jun, and W.~Bintuan.
\newblock Moment-independent importance measure of basic variable and its state
  dependent parameter solution.
\newblock {\em Structural Safety}, 38:40--47, 2012.

\bibitem{mahanta2012}
S.~Mahanta, R.~Chutia, D.~Datta, and H.~K. Baruah.
\newblock Sensitivity analysis with reference to emission concentration of
  gaussian plume model.
\newblock {\em International Journal of Energy, Information and
  Communications}, 3(2):45--52, 2012.

\bibitem{marrel2017sensitivity}
A.~Marrel, N.~Saint-Geours, and M.~De~Lozzo.
\newblock Sensitivity analysis of spatial and/or temporal phenomena.
\newblock {\em Springer Handbook on Uncertainty Quantification, Springer},
  pages 1327--1357, 2017.

\bibitem{Owen12}
A.~Owen.
\newblock Variance components and generalized {S}obol' indices.
\newblock {\em SIAM/ASA Journal on Uncertainty Quantification}, 1(1):19--41,
  2013.

\bibitem{ODC13}
A.~Owen, J.~Dick, and S.~Chen.
\newblock Higher order {S}obol' indices.
\newblock {\em Information and Inference}, 3(1):59--81, 2014.

\bibitem{owen2}
A.~B. Owen.
\newblock Better estimation of small {S}obol' sensitivity indices.
\newblock {\em ACM Transactions on Modeling and Computer Simulation (TOMACS)},
  23(2):1--17, 2013.

\bibitem{pearson1915partial}
K.~Pearson.
\newblock On the partial correlation ratio.
\newblock {\em Proceedings of the Royal Society of London. Series A, Containing
  Papers of a Mathematical and Physical Character}, 91(632):492--498, 1915.

\bibitem{pouget2016}
S.~Pouget, M.~Bursik, P.~Singla, and T.~Singh.
\newblock Sensitivity analysis of a one-dimensional model of a volcanic plume
  with particle fallout and collapse behavior.
\newblock {\em Journal of Volcanology and Geothermal Research}, 326:43--53,
  2016.

\bibitem{r2019}
{R Core Team}.
\newblock {\em R: A Language and Environment for Statistical Computing}.
\newblock R Foundation for Statistical Computing, Vienna, Austria, 2019.

\bibitem{saltelli-sensitivity}
A.~Saltelli, K.~Chan, and E.~Scott.
\newblock {\em Sensitivity analysis}.
\newblock Wiley Series in Probability and Statistics. John Wiley \& Sons, Ltd.,
  Chichester, 2000.

\bibitem{sobol2001global}
I.~Sobol.
\newblock {Global sensitivity indices for nonlinear mathematical models and
  their Monte Carlo estimates}.
\newblock {\em Mathematics and Computers in Simulation}, 55(1-3):271--280,
  2001.

\bibitem{sobol1993}
I.~M. Sobol.
\newblock Sensitivity estimates for nonlinear mathematical models.
\newblock {\em Math. Modeling Comput. Experiment}, 1(4):407--414, 1993.

\bibitem{sobolev2016}
S.~L. Sobolev.
\newblock {\em Partial Differential Equations of Mathematical Physics:
  International Series of Monographs in Pure and Applied Mathematics}.
\newblock Elsevier, 2016.

\bibitem{soetaert2012}
K.~Soetaert and F.~Meysman.
\newblock R-package reactran: Reactive transport modelling in r.
\newblock {\em Environ. Model. Softw.}, 32:49--60, 2012.

\bibitem{stockie2011}
J.~M. Stockie.
\newblock The mathematics of atmospheric dispersion modeling.
\newblock {\em Siam Review}, 53(2):349--372, 2011.

\bibitem{Sudret2008global}
B.~Sudret.
\newblock {Global sensitivity analysis using polynomial chaos expansions}.
\newblock {\em Reliability Engineering \& System Safety}, 93(7):964--979, 2008.

\bibitem{van2000asymptotic}
A.~W. van~der Vaart.
\newblock {\em Asymptotic statistics}, volume~3 of {\em Cambridge Series in
  Statistical and Probabilistic Mathematics}.
\newblock Cambridge University Press, Cambridge, 1998.

\bibitem{volpert2011}
V.~A. Volpert.
\newblock {\em Elliptic partial differential equations}, volume~1.
\newblock Springer, 2011.

\end{thebibliography}

\end{document}